\documentclass[10pt]{amsart} 
\usepackage{amsmath,amssymb,amscd,amsthm,amsfonts,amstext,amsbsy,mathrsfs,hyperref,upgreek,mathtools,stmaryrd,enumitem,bbm} 

\usepackage[textsize=footnotesize,color=blue!40, bordercolor=white]{todonotes} 

\usepackage%[scale=0.682]
[a4paper, margin=1.2in]{geometry} 

\newcommand\V{\mathsf{V}}

\renewcommand{\L}{\mathcal{L}}
\newcommand{\C}{\mathcal{C}}

\newcommand{\PP}{\mathbb{P}}
\newcommand{\QQ}{\mathbb{Q}}

\newcommand{\MM}{\mathbb{M}}

\newcommand{\anf}[1]{{\text{``}\hspace{0.3ex}{#1}\hspace{0.01ex}\text{''}}}
\newcommand{\op}{\mathsf{op}}

\newcommand{\range}{\operatorname{range}}
\newcommand{\rank}{\operatorname{rank}}
\newcommand{\Col}{\operatorname{Col}}
\newcommand{\ra}{\rightarrow}

\newcommand{\Llr}{\Longleftrightarrow}
\newcommand{\tc}{\operatorname{tc}}

\newcommand{\one}{\mathbbm1}

\newcommand{\On}{{\mathrm{Ord}}}

\newcommand{\setel}{{\prec\!\!\!\circ}}

\newcommand{\ZF}{{\sf ZF}}

\newcommand{\GB}{{\sf GB}}
\newcommand{\GBC}{{\sf GBC}}

\newcommand{\Set}[2]{\{{#1}~\vert~{#2}\}}

\newcommand{\dom}[1]{{{\rm{dom}}(#1)}}

\newtheorem{theorem}{Theorem}[section]
\newtheorem{lemma}[theorem]{Lemma}
\newtheorem{corollary}[theorem]{Corollary}

\newtheorem{question}[theorem]{Question}

\newtheorem{claim}{Claim}
\newtheorem*{claim*}{Claim}
\newtheorem*{subclaim*}{Subclaim}

\theoremstyle{definition}
\newtheorem{definition}[theorem]{Definition}
\newtheorem*{notation}{Notation}

\newtheorem{example}[theorem]{Example}
\theoremstyle{remark}
\newtheorem{remark}[theorem]{Remark}

\newenvironment{enumerate-(a)}{\begin{enumerate}[label={\upshape (\alph*)}, leftmargin=2pc]}{\end{enumerate}}

\newenvironment{enumerate-(a)-r}{\begin{enumerate}[label={\upshape (\alph*)}, leftmargin=2pc,resume]}{\end{enumerate}}

\newenvironment{enumerate-(A)}{\begin{enumerate}[label={\upshape (\Alph*)}, leftmargin=2pc]}{\end{enumerate}}

\newenvironment{enumerate-(A)-r}{\begin{enumerate}[label={\upshape (\Alph*)}, leftmargin=2pc,resume]}{\end{enumerate}}

\newenvironment{enumerate-(i)}{\begin{enumerate}[label={\upshape (\roman*)}, leftmargin=2pc]}{\end{enumerate}}

\newenvironment{enumerate-(i)-r}{\begin{enumerate}[label={\upshape (\roman*)}, leftmargin=2pc,resume]}{\end{enumerate}}

\newenvironment{enumerate-(I)}{\begin{enumerate}[label={\upshape (\Roman*)}, leftmargin=2pc]}{\end{enumerate}}

\newenvironment{enumerate-(I)-r}{\begin{enumerate}[label={\upshape (\Roman*)}, leftmargin=2pc,resume]}{\end{enumerate}}

\newenvironment{enumerate-(1)}{\begin{enumerate}[label={\upshape (\arabic*)}, leftmargin=2pc]}{\end{enumerate}}

\newenvironment{enumerate-(1)-r}{\begin{enumerate}[label={\upshape (\arabic*)}, leftmargin=2pc,resume]}{\end{enumerate}}

\begin{document}

\thanks{The first and third author were supported by DFG-grant LU2020/1-1.}



\author{Peter Holy}
\address{Peter Holy, Math. Institut, Universit\"at Bonn,
Endenicher Allee 60, 53115 Bonn, Germany}
\email{pholy@math.uni-bonn.de}
\urladdr{}

\author{Regula Krapf}
\address{Regula Krapf, Math. Institut, Universit\"at Koblenz-Landau,
Universit\"atsstra\ss e 1, 56070 Koblenz, Germany}
\email{krapf@uni-koblenz.de}
\urladdr{}

\author{Philipp Schlicht}
\address{Philipp Schlicht, Math. Institut, Universit\"at Bonn, 
Endenicher Allee 60, 53115 Bonn, Germany}
\email{schlicht@math.uni-bonn.de}
\urladdr{}

%\title{}
\date{\today}

\title{Sufficient conditions for the forcing theorem, and turning proper classes into sets}

\begin{abstract} 
We present three natural combinatorial properties for class forcing notions, which imply the forcing theorem to hold. We then show that all known sufficent conditions for the forcing theorem (except for the forcing theorem itself), including the three properties presented in this paper, imply yet another regularity property for class forcing notions, namely that proper classes of the ground model cannot become sets in a generic extension, that is they do not have set-sized names in the ground model. We then show that over certain models of G\"odel-Bernays set theory without the power set axiom, there is a notion of class forcing which turns a proper class into a set, however does not satisfy the forcing theorem. Moreover, we show that the property of not turning proper classes into sets can be used to characterize pretameness over such models of G\"odel-Bernays set theory.
\end{abstract} 

\maketitle

%\setcounter{tocdepth}{1}
%\tableofcontents 

\section{Introduction}

While the forcing theorem is a provable property of set forcing notions, this is not the case for notions of class forcing (see \cite{ClassForcing}). In this paper, we continue the work from \cite{ClassForcing} and from \cite{pretameness} by isolating further natural sufficient properties of class forcing notions that imply the forcing theorem to hold. While one of them (\emph{approachability by projections}) is only a minor generalization of the principle of the same name from \cite[Section 6]{ClassForcing}, and has a somewhat lengthy definition, the other two properties turn out to be equivalent to simple forcing properties. That is, we will show the forcing theorem to be a consequence of either not adding new sets (the \emph{set decision property}), or of every new set being added by a set-sized complete subforcing (the \emph{set reduction property}).

We then show that all of the known sufficent conditions for the forcing theorem (except for the forcing theorem itself), including the ones that we will introduce in this paper, also imply that over models of G\"odel-Bernays set theory without the power set axiom, proper classes of the ground model will not be turned into sets in a generic extension; for a notion of class forcing $\PP$, we say that \emph{a proper class $X$ of the ground model becomes a set in a $\PP$-generic extension} if there is a (set-sized) $\PP$-name $\sigma$ and a $\PP$-generic filter $G$ such that $\sigma^G=X$. Perhaps somewhat surprisingly, we then show that it is possible that a proper class \emph{can} be turned into a set, by a notion of class forcing which does not satisfy the forcing theorem. In fact, this property can even be used to characterize pretameness over certain models of G\"odel-Bernays set theory. This latter characterization continues a series of results in \cite[Theorem 1.12]{pretameness}.

We will start the paper by introducing some basic definitions and notation in the next section. This will essentially be the same basic setup as in \cite{ClassForcing} or in \cite{pretameness}.

\section{Basic Definitions and Notation}

We will work with transitive second-order models of set theory, that is models of the form $\MM=\langle M,\C\rangle$, where $M$ is transitive and denotes the collection of \emph{sets} of $\MM$, and $\C$ denotes the collection of \emph{classes} of $\MM$.\footnote{Arguing in the ambient universe $\V$, we will sometimes refer to classes of such a model $\MM$ as sets, without meaning to indicate that they are sets of $\MM$. In particular this will be the case when we talk about subsets of $M$.} We require that $M\subseteq\C$, and that elements of $\C$ are subsets of $M$. We call elements of $\C\setminus M$ \emph{proper classes} (of $\MM$).  Classical transitive first-order models of set theory are covered by this approach, letting $\C$ be the collection of classes definable over $\langle M,\in\rangle$. The theories that we will be working in will be fragments of \emph{G\"odel-Bernays set theory} $\GB$.

\begin{notation}
 \begin{enumerate-(1)}
  \item $\GB^-$ denotes the theory in the two-sorted language with variables for sets and classes, 
with the set axioms provided by the axioms of $\ZF^-$ with class parameters allowed in the schemata of Separation and Collection, and the class axioms of extensionality, foundation and first-order class comprehension (i.e.\ involving only set quantifiers).
$\GB^-$ enhanced with the power set axiom is the common collection of axioms of $\GB$. $\GBC^-$ is $\GB^-$ together with the axiom postulating the existence of a set-like well-order, i.e. a global well-order whose initial segments are set-sized. % Note that by \cite[Remark 1.2]{pretameness}, this is equivalent to the existence of a global well-order of order-type $\On^M$.
  \item By a \emph{countable} transitive model %of $\GB^-$, $\GB$ or $\GBC^-$
, we mean a transitive second-order model $\MM=\langle M,\C\rangle$ % of $\GB^-$, $\GB$ or $\GBC^-$ respectively,
such that both $M$ and $\C$ are countable in $\V$. 
%  \item Given a transitive second-order model of the form $\MM=\langle M,\C\rangle$, we let $\Def(\MM)$ denote the collection of subsets of $M$ that are first-order definable over $M$ using class parameters from $\C$ as predicates. Note that the axiom of first-order class comprehension implies that if $\MM=\langle M,\C\rangle\models\GB^-$, then $\C$ is closed under first-order definability (over $\MM$), that is $\Def(\MM)=\C$.
 \end{enumerate-(1)}
\end{notation}

Fix a countable transitive model $\MM=\langle M,\C\rangle$ of $\GB^-$. By a \emph{notion of class forcing} (for $\MM$) we mean a partial order $\PP=\langle P,\leq_\PP\rangle$ such that $P,\leq_\PP\,\in\C$. 
We will frequently identify $\PP$ with its domain $P$. In the following, we also fix a notion of class forcing $\PP=\langle P,\leq_\PP\rangle$ for $\MM$.  

We call $\sigma$ a \emph{$\PP$-name} if all elements of $\sigma$ are of the form $\langle\tau,p\rangle$, where 
$\tau$ is a $\PP$-name and $p\in\PP$. 
We define $M^\PP$ to be the set of all $\PP$-names that are elements of $ M$ and define $\C^\PP$ to be the set of all $\PP$-names that are elements of $\C$.
In the following, we will usually call the elements of $M^\PP$ simply \emph{$\PP$-names} and we will call the elements of $\C^\PP$ \emph{class $\PP$-names}.
%\todo[inline]{Should we say here what a $\PP$-name is? Or maybe already in the introduction. Of course it doesn't really matter, but then we talk about the definability of the class of $\PP$-names...}
If $\sigma\in M^\PP$ is a $\PP$-name, we define 
$$\rank\sigma=\sup\{\rank\tau+1\mid\exists p\in\PP\,[\langle\tau,p\rangle\in\sigma]\}$$ 
to be its \emph{name rank}. 
%We will sometimes also need to use the usual set theoretic rank of some $\sigma\in M$, which we will denote as $\rnk(\sigma)$.

We say that a filter $G$ on $\PP$ is \emph{$\PP$-generic over $\MM$}, if $G$ meets every dense subset of $\PP$ that is an element of $\C$. 
Given such a filter $G$ and a $\PP$-name $\sigma$, we recursively define the \emph{$G$-evaluation} of $\sigma$ as
$$\sigma^G=\{\tau^G\mid\exists p\in G\,[\langle\tau,p\rangle\in\sigma]\},$$
and similarly we define $\Gamma^G$ for $\Gamma\in\C^\PP$. Moreover, if $G$ is $\PP$-generic over $\MM$, then we set
$M[G]=\Set{\sigma^G}{\sigma\in M^\PP}$ and $\C[G]=\Set{\Gamma^G}{\Gamma\in\C^\PP}$. 

Given an $\L_\in$-formula $\varphi(v_0,\dots,v_{m-1},\vec\Gamma)$, where $\vec\Gamma\in(\C^\PP)^n$ are class name parameters, $p\in\PP$ and $\vec\sigma\in(M^\PP)^m$, we write
$p\Vdash_\PP^\MM\varphi(\vec\sigma,\vec\Gamma)$, if for every $\PP$-generic filter $G$ over $\MM$ with $p\in G$, $\langle M[G],\Gamma_0^G,\dots,\Gamma_{n-1}^G\rangle\models\varphi(\sigma_0^G,\dots,\sigma_{m-1}^G,\Gamma_0^G,\dots,\Gamma_{n-1}^G)$. 

% \begin{definition}\label{def:ft}
%  Let $\varphi\equiv\varphi(v_0,\ldots,v_{m-1})$ be an $\L^n$-formula. 
%  \begin{enumerate-(1)}
%   \item We say that \emph{$\PP$ satisfies the definability lemma for $\varphi$ over $\MM$} if 
%    \begin{equation*}
%   \Set{\langle p,\sigma_0,\ldots,\sigma_{m-1}\rangle\in P\times M^\PP\times\ldots\times M^\PP}{p\Vdash^{\MM,\vec{\Gamma}}_\PP\varphi(\sigma_0,\ldots,\sigma_{m-1})}\in\C 
%  \end{equation*} 
%  for all $\vec{\Gamma}\in(\C^\PP)^n$. 
%  
%  \item We say that \emph{$\PP$ satisfies the truth lemma for $\varphi$ over $\MM$} if for all $\sigma_0,\ldots,\sigma_{m-1}\in M^\PP$, $\vec{\Gamma}\in(\C^\PP)^n$ and every filter $G$ which is $\PP$-generic over $\MM$ with 
%   \begin{equation*}
%     M_{\vec{\Gamma}}[G]\models\varphi(\sigma_0^G,\ldots,\sigma_{m-1}^G), 
%   \end{equation*}
%   there is $p\in G$ with $p\Vdash^{\MM,\vec{\Gamma}}_\PP\varphi(\sigma_0,\ldots,\sigma_{m-1})$.
% 
%   \item We say that \emph{$\PP$ satisfies the forcing theorem for $\varphi$ over $\MM$} if $\PP$ satisfies both the definability lemma and the truth lemma for $\varphi$ over $\MM$.
%  \end{enumerate-(1)}
%  \end{definition}

A fundamental result in the context of set forcing is the \emph{forcing theorem}. It consists of two parts, the first one of which, the so-called \emph{definability lemma}, states that the forcing relation is definable in the ground model, and the second part, denoted as the \emph{truth lemma}, says that every formula which is true in a generic extension $M[G]$ is forced by some condition in the generic filter $G$. In the context of second-order models of set theory, this has the following natural generalization:

\begin{definition}\label{def:ft}
 Let $\varphi\equiv\varphi(v_0,\ldots,v_{m-1},\vec\Gamma)$ be an $\L_\in$-formula with class name parameters $\vec\Gamma\in(\C^\PP)^n$. 
 \begin{enumerate-(1)}
  \item We say that \emph{$\PP$ satisfies the definability lemma for $\varphi$ over $\MM$} if 
   \begin{equation*}
  \Set{\langle p,\sigma_0,\ldots,\sigma_{m-1}\rangle\in P\times(M^\PP)^m}{p\Vdash^\MM_\PP\varphi(\sigma_0,\ldots,\sigma_{m-1},\vec\Gamma)}\in\C. 
 \end{equation*} 
 \item We say that \emph{$\PP$ satisfies the truth lemma for $\varphi$ over $\MM$} if for all $\sigma_0,\ldots,\sigma_{m-1}\in M^\PP$, and every filter $G$ which is $\PP$-generic over $\MM$ with 
  \begin{equation*}
    \langle M[G],\Gamma_0^G,\dots,\Gamma_{n-1}^G\rangle\models\varphi(\sigma_0^G,\ldots,\sigma_{m-1}^G,\Gamma_0^G,\dots,\Gamma_{n-1}^G), 
  \end{equation*}
  there is $p\in G$ with $p\Vdash^{\MM}_\PP\varphi(\sigma_0,\ldots,\sigma_{m-1},\vec\Gamma)$.
  \item We say that \emph{$\PP$ satisfies the forcing theorem for $\varphi$ over $\MM$} if $\PP$ satisfies both the definability lemma and the truth lemma for $\varphi$ over $\MM$.
 \end{enumerate-(1)}
 \end{definition}
 
Note that in class forcing, the forcing theorem may fail even for atomic formulae (\cite[Theorem 1.3]{ClassForcing}). 
A crucial result is that if the definability lemma holds for one atomic formula, then the forcing theorem holds for each $\L_\in$-formula with class name parameters (see \cite[Theorem 4.3]{ClassForcing}). %In Sections \ref{sdp}, \ref{srp} and \ref{approachability}, we will present three properties of class-sized forcing notions which imply the forcing theorem to hold.

%Another property which implies the forcing theorem to hold is the following:
\begin{definition}\cite[Chapter 2]{MR1780138}
A notion of (class) forcing $\PP$ for $\MM$ is \emph{pretame} for $\MM$ if for every $p\in\PP$ and for every sequence of dense subclasses $\langle D_i\mid i\in I\rangle\in\C$ of $\PP$ with $I\in M$, there is $q\leq_\PP p$ and $\langle d_i\mid i\in I\rangle\in M$ such that for every $i\in I$, $d_i\subseteq D_i$ and $d_i$ is predense below $q$.
\end{definition}
The observation that pretame notions of class forcing satisfy the forcing theorem over any model of $\GB^-$ was first made by Maurice Stanley in his PhD thesis (\cite{stanley_thesis}), see also \cite{stanley_classforcing}. Moreover, Stanley observed (\cite{stanley_thesis},\cite{stanley_classforcing}) that pretameness characterizes the preservation of $\GB^-$ over models of $\GB$. For a proof of both these results in our setting, consult \cite[Theorem 2.4 and Theorem 3.1]{pretameness}.

\medskip

In this paper, we will frequently make use of a particular collection of notions of class forcing:
For an ordinal $\gamma\in\On^M$ and a proper class $Y\in\C$, let $\Col(\gamma,Y)^M$ denote the forcing notion that adds a surjection from $\gamma$ to $Y$ with conditions of size less than the $M$-cardinality of $\gamma$ over $\MM$, that is the conditions of $\Col(\gamma,Y)^M$ are partial functions from $\gamma$ to $Y$ with domain of size less than the cardinality of $\gamma$ in $M$, ordered by reverse inclusion.
A variant of $\Col(\gamma,Y)^M$ is provided by the forcing notion $\Col_*(\gamma,Y)^M$, which consists of those conditions $p\in\Col(\gamma,Y)^M$ whose domain is an ordinal. In Sections \ref{sdp} and \ref{approachability} we show that for every $\gamma\in\On^M$ and for every $Y\in\C$, both $\Col(\gamma,Y)^M$ and $\Col_*(\gamma,Y)^M$ satisfy the forcing theorem over $\MM$. \footnote{Note that for $\gamma=\omega$ and $Y=\On^M$, this is verified in \cite[Proposition 2.25]{MR1780138}, and also follows from more general results in \cite[Lemma 2.2, Lemma 6.3 and Theorem 6.4]{ClassForcing}.}

%In Section \ref{sdp} we prove that for every $\gamma\in\On^M$ and for every class $Y\in\C$, $\Col_*(\gamma,Y)$ satisfies the forcing theorem. Note that for $\gamma=\omega$ and $Y=\On^M$ it follows from \cite[Lemma 2.2, Lemma 6.3 and Theorem 6.4]{ClassForcing} that $\Col(\gamma,Y)$ and $\Col_*(\gamma,Y)$ satisfy the forcing theorem.
%
%\todo[inline]{Generalize this to $\Col_*(\gamma,\On)$}
%Furthermore, let $\Col_*(\omega,\On)^M$ denote the notion of forcing with conditions of the form $p\colon n\to\On^M$ for $n\in\omega$, ordered by reverse inclusion. Any generic for this forcing clearly gives rise to a cofinal sequence from $\omega$ to $\On^M$, this forcing does not add any new sets (i.e.\ $M[G]=M$ whenever $G$ is $\Col_*(\omega,\On)^M$-generic over $M$) and it satisfies the forcing theorem. The latter (easy) facts were verified in \cite[Proposition 2.25]{MR1780138} and follow from more general results in \cite[Lemma 2.2, Lemma 6.3 and Theorem 6.4]{ClassForcing}. They also follow from the results of Section \ref{sdp} of the present paper.

\section{The Set Decision Property}\label{sdp}

In this section, we introduce a simple combinatorial property which implies the forcing theorem. Moreover, we will show that this property exactly characterizes those notions of class forcing which do not add any new sets.

\begin{definition}
Let $\MM=\langle M,\C\rangle$ be a countable transitive model of $\GB^-$ and let $\PP$ be a notion of class forcing for $\MM$. Let $\dot G$ denote the canonical $\PP$-name for the generic filter.
\begin{enumerate-(1)-r}
 \item If $p\in\PP$ and $\sigma$ is a $\PP$-name, then we define the \emph{$p$-evaluation} of $\sigma$ by 
$$\sigma^p=\{\tau^p\mid\exists q\in\PP\,[\langle\tau,q\rangle\in\sigma\wedge\forall r\leq_\PP p\,( r\parallel_\PP q)]\}.$$
 \item Given conditions $p$ and $q$ in $\PP$, we write $p\leq_\PP^* q$ iff $\forall r\leq_\PP p\,(r\parallel_\PP q)$ (equivalently, $p\Vdash_\PP q\in\dot G$). Note that if $\PP$ is separative, then $p\leq_\PP^* q$ if and only if $p\leq_\PP q$.
 \item If $A\subseteq\PP$ is a set of conditions and $p\in\PP$, we write $p\bot_\PP A$ or $p\leq_\PP^*A$ if $\forall a\in A\,(p\bot_\PP a)$ or $\forall a\in A\,(p\leq_\PP^*a)$ respectively.
 \item If $A\subseteq\PP$ is a set of conditions and $p\in\PP$, then $p$ \emph{decides} $A$ (we write $p\sim_\PP A$) if for every $a\in A$, either $p\leq_\PP^*a$ or $p\bot_\PP a$.
 \item We say that $\PP$ has the \emph{set decision property} if
for every $p\in\PP$ and every set $A\subseteq\PP$ in $M$, there is an extension $q\leq_\PP p$ of $p$	 such that 
$q$ decides $A$.
\end{enumerate-(1)-r}
Note that if $p$ decides $A$, then $p$ decides for every condition in $A$ whether it lies in the generic filter or not, i.e.\ $p$ decides $\dot G\cap A$. 
%We abbreviate this by writing $q\parallel A$.
\end{definition}

A natural example of forcing notions with the set decision property are the (strategically) ${<}\On$-closed forcing notions,
i.e.\ notions of forcing which are ${<}\kappa$-(strategically) closed for every cardinal $\kappa$. We will leave the adaption of the example below to the case of strategically ${<}\On$-closed notions of forcing (and also the task of giving a precise definition of this property) to the reader, as it is straightforward and we will not make use of any such property in this paper.

\begin{example}
  Assuming that $\MM$ is a model of $\GBC^-$, then every ${<}\On$-closed notion of class forcing $\PP$ for $\MM$ has the set decision property: Let $p\in\PP$ and let $A\subseteq\PP$ be a set of conditions. Using choice, we can enumerate $A$ as $\{a_i\mid i<\kappa\}$ for some cardinal $\kappa$. 
Inductively, we define a sequence $\langle p_i\mid i<\kappa\rangle$ of conditions such that 
for every $i<j<\kappa$, $p_j\leq_\PP a_i$, or $p_j\bot_\PP a_i$. 
\begin{itemize}
 \item Let $p_0=p$.
 \item Assume that $p_i$ has already been defined. If $p_i\parallel_\PP a_i$, then pick $p_{i+1}$ that is stronger than both $p_i$ and $a_i$, using the existence of a global well-order.
Otherwise, put $p_{i+1}=p_i$.
 \item For a limit ordinal $\alpha$, we use ${<}\On$-closure of $\PP$ and the global well-order to pick $p_\alpha$ stronger than $p_i$ for all $i<\alpha$. 
\end{itemize}
Now let $q\in\PP$ be a condition stronger than every $p_i$ for $i<\kappa$. By construction, $q$ decides $A$. 
\end{example}

\begin{example}Let $\gamma\in\On^M$ and $Y\in\C$ a proper class. Then the forcing notion $\PP=\Col_*(\gamma,Y)^M$ satisfies the set decision property: Suppose that $A\in M$ is a subset of $\PP$ and $p\in\PP$. Let $\beta=\dom{p}\in\On^M$. Now note that by assumption, $X=\bigcup_{q\in A}\range(q)\in M$ and since $Y$ is a proper class, there is $y\in Y\setminus X$.
 Then $q=p\cup\{\langle \beta,y\rangle\}$ decides $A$.
\end{example}

\begin{definition}
  Given a notion of class forcing $\PP$ and a $\PP$-name $\sigma$, we define the conditions appearing in (the transitive closure of) $\sigma$ by induction on name rank as
\[\tc(\sigma)=\bigcup\{\{p\}\cup\tc(\tau)\mid\langle\tau,p\rangle\in\sigma\}.\]
\end{definition}

\begin{lemma}\label{lem:ft}
Every class forcing $\PP$ for $\MM$ with the set decision property satisfies the forcing theorem and does not add new sets, that is $M[G]=M$ whenever $G$ is $\PP$-generic over $\MM$. 
\end{lemma}

\begin{proof}
 By \cite[Theorem 4.3]{ClassForcing}, to verify the forcing theorem it is enough to check that the definability lemma holds for $\anf{v_0=v_1}$. 
Let $\sigma,\tau\in M^\PP$. %We show how to define $\{p\Vdash_\PP\sigma=\tau$.
Let $A=\tc(\sigma\cup\tau)$ and let $p\in\PP$. 
%By assumption there is $q\leq p$ such that for all $a\in A$, $q\leq a$ or $q\bot a$. 
Then by the set decision property $p\Vdash_\PP\sigma=\tau$ if and only if $\forall q\leq_\PP p\,(q\sim_\PP A \ra q\Vdash_\PP\sigma=\tau)$. 
But if $q\sim_\PP A$ and $q\in G$ then $\sigma^q=\sigma^G$ (this in particular implies that $\sigma^G\in M$ and hence that $\PP$ does not add new sets), thus
%If $\sigma^q=\{\mu^q\mid\exists r(\langle\mu,r\rangle\in\sigma\wedge q\leq r)\}$ 
%(REPLACE WITH PREVIOUSLY DEFINED PARTIAL EVALUATION), then ... ($\sigma^G$ can be calculated from $\sigma^q$ in $M$!). 
we obtain $q\Vdash_\PP\sigma=\tau$ iff $\sigma^q=\tau^q$. Consequently, $p\Vdash_\PP\sigma=\tau$ can be defined by
$\forall q\leq_\PP p\,(q\sim_\PP A\ra\sigma^q=\tau^q)$. 
%(I DON'T THINK THIS IS TRUE, but what we have to do is calculate $\sigma^G$ and $\tau^G$, which we can by the above).
%Similarly, one defines $p\Vdash_\PP\sigma\in\tau$ iff $\forall q\leq p(q\parallel A\ra\sigma^q\in\tau^q)$.
%and for class names $\Gamma$, $p\Vdash_\PP\sigma\in\Gamma$ iff $p\Vdash_\PP\sigma\in\{\langle\mu,p\rangle\in\Gamma\mid\rank\mu\leq\rank\sigma\}$.
%Next, we have to prove that the truth lemma holds. Since all cases are similar, we 
%just prove it for equations. Assume that $G$ is generic and that $\sigma^G=\tau^G$. 
%Since by assumption, \[D=\{p\in\PP\mid\forall q\in\PP\cap\tc(\sigma\cup\tau)\ (p\leq q\vee p\bot q)\}\] is dense, there is 
%$p\in G\cap D$. Thus since $\sigma^p=\sigma^G=\tau^G=\tau^p$, we have $p\Vdash_\PP\sigma=\tau$. 
\end{proof}

% \begin{remark}
%  Property \eqref{al:str or inc} is equivalent to the following condition: $\PP$ adds no sets and the forcing 
% theorem holds for all formulae of the form $\sigma=\check{x}$. 
% \end{remark}
% 
% \begin{proof}
% The first direction is clear. For the converse suppose that $\PP$ adds no sets and that the forcing theorem 
% holds for all formulae of the form $\sigma=\check{x}$. Now let $p\in\PP$ and $A\subseteq\PP$ a set. 
% Now put 
% $$\sigma=\{\langle\check{a},a\rangle\mid a\in A\}.$$
% Clearly, $1_\PP\Vdash_\PP\sigma\subseteq\check{A}$. Moreover, we claim that the class
% $$D=\{q\in\PP\mid\exists x(q\Vdash_\PP\sigma=\check{x})\}$$
% is dense. For this let $q\in\PP$. Now if $G$ is generic such that $q\in G$, then there must be some 
% $x\in M$ such that $M[G]\models\sigma^G=x$. But by the truth lemma for $\sigma=\check{x}$ there has to 
% be $r\in G$ such that $r\Vdash_\PP\sigma=\check{x}$. Since $q$ and $r$ are compatible, there is $s\leq q,r$ 
% and $s\in D$. Now by density of $D$ we can take $q\leq p$ and $x\in M$ such that 
% $q\Vdash_\PP\sigma=\check{x}$. Clearly, $x\subseteq A$ and for all $a\in A$, if $a\in x$ then $q\leq a$ and 
% if $a\notin x$ then $q\bot a$. 

\begin{lemma}\label{lem:no sets->sdp} 
 Let $\PP$ be a notion of class forcing for $\MM$ which adds no new sets. Then $\PP$ has 
the set decision property. 
\end{lemma}

\begin{proof}
Let $A\subseteq\PP$ be a set of conditions in $M$ and let $p\in\PP$. 
We have to find $q\leq_\PP p$ such that $q\sim_\PP A$. Assume for a contradiction that no such $q$ exists. 

Enumerate (in $\V$) all elements of $\C$ that are dense subsets of $\PP$ by $\langle D_n\mid n\in\omega\rangle$ and 
all subsets of $A$ which are elements of $M$ by $\langle x_n\mid n\in\omega\rangle$.
Let $\sigma=\{\langle\check{a},a\rangle\mid a\in A\}$.
We will find a $\PP$-generic filter $G$ such that $\sigma^G\not\in M$, which clearly contradicts our assumption on $\PP$. 
For this we define a decreasing sequence of conditions $\langle q_n\mid n\in\omega\rangle$ below $p$ and a sequence
$\langle a_n\mid n\in\omega\rangle$ of conditions in $A$. Let $q_0=p$. Given $q_n$, 
note that by our assumption it cannot be the case that $q_n\leq_\PP^* x_n$ and $q_n\bot_\PP(A\setminus x_n)$. Hence there is $a_n\in A$ such that either $a_n\in x_n$ and $q_n\not\leq_\PP^* a_n$ or 
$a_n\notin x_n$ and $q_n\parallel_\PP a_n$. In the first 
case we pick $r\leq_\PP q_n$ such that $r\bot_\PP a_n$. In the second case, we strengthen $q_n$ 
to $r\leq_\PP q_n,a_n$. Now take $q_{n+1}\leq_\PP r$ such that $q_{n+1}\in D_n$. 
Finally, this means that $G=\{q\in\PP\mid\exists n\in\omega\,(q_n\leq_\PP q)\}$ is a generic filter. 
But since $\PP$ doesn't add new sets and since $\one_\PP\Vdash_\PP\sigma\subseteq\check{A}$, there must be 
some $n\in\omega$ such that $\sigma^G=x_n$. But we have that either $a_n\in x_n$ and $a_n\bot_\PP q_{n+1}$, thus
$a_n\notin\sigma^G$, or $a_n\notin x_n$ but $q_{n+1}\leq_\PP a_n$ implying that $a_n\in\sigma^G$. We have thus reached a contradiction.
\end{proof}

Putting together Lemmata \ref{lem:ft} and \ref{lem:no sets->sdp} we obtain

\begin{corollary}\label{thm:sep no sets->ft}
 Every class forcing which does not add new sets satisfies the forcing theorem. 
\end{corollary}

%\begin{definition}
%We say that a model $\langle M,\C\rangle$ of $\GB^-$ satisfies \emph{representatives choice}, if for every equivalence relation 
%$E\in\C$ there is $A\in\C$ and a surjective map $\pi:\dom E\ra A$ in $\C$ such that 
%$\langle x,y\rangle\in E$ if and only if $\pi(x)=\pi(y)$. 
%\end{definition}

%Representatives choice is an easy consequence of either global choice or the power set axiom (see \cite[Section 3]{ClassForcing} for a brief argument concerning the %latter), and thus of $\GB$. In \cite[Section 5]{ClassForcing}, we show how representatives choice allows one to pass from $\PP$ to its separative quotient %$\Ss(\PP)=\{[p]\mid p\in\PP\}$, using a suitable definition of $[p]$, where we equip $\Ss(\PP)$ with the usual ordering given by 
%$$[p]\leq[q]\quad\text{iff}\quad\forall r\in\PP(r\parallel p\ra r\parallel q).$$
%Clearly, the canonical map $\PP\ra\Ss(\PP)$ is a surjective complete embedding. 
%Now if $\PP$ adds no sets, then $\Ss(\PP)$ doesn't add any sets either and thus by Corollary \ref{thm:sep no sets->ft}, $\Ss(\PP)$ satisfies 
%the forcing theorem. Thus $\PP$ satisfies the forcing theorem by stipulating
%$p\Vdash_\PP\varphi$ iff $[p]\Vdash_{\Ss(\PP)}\varphi^{\Ss(\PP)}$, where $\varphi^{\Ss(\PP)}$ is defined 
%from $\varphi$ by replacing every $\PP$-name $\sigma$ by $\sigma^{\Ss(\PP)}=\{\langle\tau^{\Ss(\PP)},[q]\rangle\mid\langle\tau,q\rangle\in\sigma\}$. 

%Thus we have shown

%\begin{theorem}
%Representatives choice implies that every class forcing which adds no sets satisfies the forcing theorem. 
%\end{theorem}

In a series of two blog posts (\cite{gitman}), Victoria Gitman claims to show (as a result of discussions with Joel Hamkins) that class forcing that does not add new sets always satisfies the forcing theorem, at least in case the notion of forcing under consideration is ${<}\On$-strategically closed. However in her setup, she allows only for canonical names (or in fact, constant symbols) for sets (and her class names are just classes of pairs of canonical names and forcing conditions), making the property of not adding new sets a trivial consequence. The problem with her approach is however that one cannot actually prove that a given notion of class forcing does not add new sets in her context (even if it is ${<}\On$-closed), or to put it differently, to apply Gitman's results to a specific notion of forcing, one would first need to go through an argument, in the standard setting that allows for the usual names for sets, that the notion of forcing under discussion does not add new sets, thus 
verifying that the usual notion of generic extension corresponds with the one used by Gitman in this case. %To perform such an argument, one would need to go through arguments similar to ours in this chapter.% Moreover, Gitman's approach is limited to notions of forcing that are ${<}\On$-strategically closed.

\section{The Set Reduction Property}\label{srp}

In this section, we introduce a weakening of the set decision property, that we call the \emph{set reduction property}, and verify that it is equivalent to the property that every new set added by $\PP$ is already added by some set-size complete subforcing of $\PP$, and moreover that it still implies the forcing theorem to hold for $\PP$.

\medskip

Let $\MM=\langle M,\C\rangle$ be a fixed countable transitive model of $\GB^-$.

\medskip

\begin{notation}Let $\PP$ be a notion of class forcing for $\MM$.
	\begin{enumerate-(1)}
		\item 
We let $\QQ\setel\PP$ denote the statement that $\QQ$ is a set-sized complete subforcing of $\PP$.
\item Given $p\in\PP$ and $\QQ\setel\PP$, let $\QQ^{\parallel p}$ be the set of conditions in $\QQ$ that are compatible with $p$ in $\PP$.
\item
We say that \emph{every new set added by $\PP$ is added by a set-sized complete subforcing of $\PP$} if whenever $G$ is $\PP$-generic over $\MM$ and $x\in M[G]\setminus M$, then there is $\QQ\setel\PP$ such that $x$ is already an element of the induced $\QQ$-generic extension $M[\bar G]$ of $M$, where $\bar G=G\cap\QQ$.
\end{enumerate-(1)}
\end{notation}

 We will show that any $\PP$ with the property that every new set added by $\PP$ is added by a set-sized complete subforcing of $\PP$ satisfies the forcing theorem, improving our result on the set decision property from Section \ref{sdp}, and also generalizing a classical result of Zarach (\cite{MR0345819}), where he showed that any notion of forcing that is the $\On^M$-length union of complete set-sized subforcings satisfies the forcing theorem.\footnote{A generalization of this result in a different direction was also obtained in \cite[Section 6]{ClassForcing}.} 

\begin{definition}
Let $\PP$ be a  notion of class forcing for $\MM$. We  say that $\PP$ satisfies the \emph{set reduction property (over $\MM$)} if whenever $A\subseteq\PP$ is a set (in $M$) and $p\in\PP$, then there is $q\leq_\PP p$ and $\QQ\setel \PP$ (in $M$) such that $(*)(A,q,\QQ)$ holds: for all $a\in A$, $\{r\in\QQ^{\parallel q}\mid \forall s\leq_\PP q,r\ (s\leq_\PP^* a)$ or $\forall s\leq_\PP q,r\ (s\perp_\PP a)\}$ is dense in $\QQ^{\parallel q}$.
\end{definition}

\begin{remark}
 The set decision property implies the set reduction property, as is witnessed by the trivial forcing notion.
 \end{remark}

\begin{definition}
Given a notion of class forcing $\PP$ for $\MM$, $\sigma\in M^\PP$, $\QQ\setel\PP$ and $q\in\PP$, we define a $\QQ$-name $\sigma_q^\QQ$, the \emph{$q$-reduction of $\sigma$ to $\QQ$}, by recursion as follows. $$\sigma_q^\QQ=\{\langle\tau_q^\QQ,r\rangle\mid r\in\QQ\wedge \exists a\,[\langle\tau,a\rangle\in\sigma\wedge\forall s\leq_\PP q,r\,(s\leq_\PP^*a)]\}$$
\end{definition}

\begin{definition}
  Given a notion of class forcing $\PP$ and a $\PP$-name $\sigma$, we define the conditions appearing in (the transitive closure of) $\sigma$ by induction on name rank as
\[\tc(\sigma)=\bigcup\{\{p\}\cup\tc(\tau)\mid\langle\tau,p\rangle\in\sigma\}.\]
\end{definition}

\begin{lemma}\label{lemma:reduced names}
Suppose that $\PP$ is a notion of class forcing for $\MM$, $q\in\PP$ and $\QQ\setel\PP$, suppose that $(\ast)(A,q,\QQ)$ holds and let $G$ be $\PP$-generic with $q\in G$. Then for every $\sigma\in M^\PP$ with $\tc(\sigma)\subseteq A$, $\sigma^G=(\sigma_q^\QQ)^{\bar G}$, where $\bar G=G\cap\QQ$. 
\end{lemma}

\begin{proof}
We proceed by induction on the name rank of $\sigma$. Suppose that $\tau^G\in\sigma^G$, because there is $a\in G$ so that  $\langle\tau,a\rangle\in\sigma$. Using $(\ast)(A,q,\QQ)$, we can find a condition $r\in\bar G$ such that for all $s\leq_\PP q,r$, it holds that $s\leq^*_{\PP}a$. Then $\langle\tau_q^\QQ,r\rangle\in\sigma_q^\QQ$ and by induction, $(\tau_q^\QQ)^{\bar G}=\tau^G$, hence $\tau^G\in(\sigma_q^{\QQ})^{\bar G}$.
If on the other hand $(\tau_q^{\QQ})^{\bar G}\in(\sigma_q^{\QQ})^G$, because there is $r\in\bar G$ such that $\exists a\,\langle\tau,a\rangle\in\sigma\,\land\,\forall s\le_{\PP}q,r\ s\le^*_{\PP}a$, then inductively $\tau^G=(\tau_q^{\QQ})^{\bar G}\in\sigma^G$.
\end{proof}

\begin{lemma}\label{srp1}
  Every notion of class forcing $\PP$ for $\MM$ with the set reduction property satisfies the forcing theorem.
\end{lemma}
\begin{proof}
  Let $\PP$ be a notion of class forcing for $\MM$ with the set reduction property.
  We show that $\{\langle p,\sigma,\tau\rangle\in M\mid p\Vdash_\PP\sigma=\tau\}\in\C$, which suffices by \cite[Theorem 4.3]{ClassForcing}. Fix $\PP$-names $\sigma$ and $\tau$ and let $A=\tc(\sigma\cup\tau)$.% and let $Q\setel P$ be the corresponding set provided by the set decision property.
  \begin{claim}
    $p\Vdash_\PP\sigma=\tau\iff\forall q\leq_\PP p\ [\exists \QQ\setel \PP\,(*)(A,q,\QQ)\to q\Vdash_\PP\sigma=\tau]$.
  \end{claim}
  \begin{proof}
    The left to right direction is immediate. For the right to left direction, note that $D=\{q\leq_\PP p\mid\exists \QQ\setel \PP\,(*)(A,q,\QQ)\}\in\C$ by first order class comprehension, and that $D$ is dense below $p$ as a direct consequence of the set reduction property.
  \end{proof}
%   Now given $q\in\QQ$, we want to \emph{reduce} $\sigma$ and $\tau$ to $Q$-names $\bar\sigma_q$ and $\bar\tau_q$ respectively. For this we inductively define the operation $\pi\to\bar\pi_q$ for any $P$-name $\pi$ with $P\cap\tc(\pi)\subseteq A$. Namely, whenever the pair $(\rho,a)$ appears as an element of $\pi$, we remove it and instead put every element of $\{(\bar\rho_q,r)\mid q\,\land\,r\le a\}$ into $\pi$. \footnote{We could at this point rather use an antichain of such $r$, this would perhaps allow us to use forcings with the $\On$-cc, rather than just set-sized ones?}
  \begin{claim}
    Assume that  $\tc(\sigma)\cup\tc(\tau)\subseteq A$ and $(*)(A,q,\QQ)$ holds. Then $q\Vdash_\PP\sigma=\tau$ if and only if
    \[\forall r_0\in \QQ^{\parallel q}\,\exists r_1\in\QQ^{\parallel q}(r_1\le_{\QQ}r_0\,\land\,r_1\Vdash_\QQ\sigma_q^\QQ=\tau_q^\QQ.)\]
  \end{claim}
  \begin{proof}
    For the forward direction, assume that $q\Vdash_\PP\sigma=\tau$ and that $r_0\in \QQ^{\parallel q}$. Let $G$ be $\PP$-generic with $r_0,q\in G$, hence $\sigma^G=\tau^G$. Let $\bar G$ denote the $\QQ$-generic induced by $G$, that is $\bar G=G\cap\QQ$. By Lemma \ref{lemma:reduced names}, $(\sigma_q^\QQ)^{\bar G}=(\tau_q^\QQ)^{\bar G}$. Let $r_1\leq_\QQ r_0$ be a condition in $\bar G$ forcing this. Then $r_1\in Q^{\parallel q}$ and $r_1\Vdash_{\QQ}\bar\sigma_q=\bar\tau_q$.
%     We show that $(\sigma_q^\QQ)^{\bar G}=(\tau_q^\QQ)^{\bar G}$. 
%     Suppose that $\langle\rho,a\rangle\in\sigma$ and $r\in\bar G$ such that for all $s\leq_\PP q,r$, $s\leq_\PP^*a$. Since both $q$ and $r$ are in $G$ there is $s\in G$ with $s\leq_\PP q,r$. Then $s\leq_\PP^* a$. But then $a\in G$ and so $\rho^G\in\sigma^G=\tau^G$. This means that there must be $\langle\mu,b\rangle\in\tau$ such that $b\in G$ and $\mu^G=\rho^G$. By induction, we have that $(\mu_q^\QQ)^{\bar G}=(\rho_q^\QQ)^{\bar G}$. Using $(\ast)(A,q,\QQ)$ there must be some $r\in\bar G$ such that for all $s\leq_\PP q,r(s\leq_\PP^*b)$. Then $(\mu_q^\QQ)^{\bar G}\in(\tau_q^\QQ)^{\bar G}$ as desired. 
    
    For the backward direction, suppose that the right-hand side holds and let $G$ be $\PP$-generic with $q\in G$. Let $\bar G=G\cap\QQ$. Take $r\in\bar G$ with $r\Vdash_\QQ\sigma_q^\QQ=\tau_q^\QQ$. 
    Then $\sigma^G=\tau^G$ by Lemma \ref{lemma:reduced names}. 
    %As before, it follows that $\sigma^G=\tau^G$. 
    Since $G$ was arbitrary, this means that $q\Vdash_\PP\sigma=\tau$. 
%     By the definition of $\bar\sigma_q$ and $\bar\tau_q$, it follows that $\bar\sigma_q^{\bar G}=\bar\tau_q^{\bar G}$. Let $r_1\le r_0$, $r_1\in\bar G$ be a condition forcing this. Then $r_1\in Q^{\parallel q}$ and $r_1\Vdash_Q\bar\sigma_q=\bar\tau_q$.
%     
%     For the right to left direction, let $G$ be $P$-generic with $q\in G$. Let $r\in\bar G$ be such that $r\Vdash_Q\bar\sigma_q=\bar\tau_q$. Then $\bar\sigma_q^{\bar G}=\bar\tau_q^{\bar G}$. It follows that $\sigma^G=\tau^G$, hence $q\Vdash\sigma=\tau$.
  \end{proof}

Note that since $\QQ$ is a notion of set forcing, it satisfies the forcing theorem, and thus the $\QQ$-forcing relation is definable over $M$. Using the above claims, it is immediate that $\{\langle p,\sigma,\tau\rangle\in M\mid p\Vdash_\PP\sigma=\tau\}$ is definable over $\langle M,\PP,\le_\PP\rangle$, and is thus an element of $\C$.
\end{proof}

\begin{lemma}\label{srp2}
 Let $\PP$ be a notion of class forcing for $\MM$. Then $\PP$ has the set reduction property if and only if every new set added by $\PP$ is added by a set-sized complete subforcing of $\PP$. 
\end{lemma}

\begin{proof}
The forward direction is immediate by Lemma \ref{lemma:reduced names}.
For the backward direction, let $A\subseteq\PP$ be a set of conditions and let $\sigma=\{\langle\check{a},a\rangle\mid a\in A\}$.
Assume that every new set added by $\PP$ is added by a set-sized complete subforcing of $\PP$. However, suppose for a contradiction that $\PP$ does not have the set reduction property, as is witnessed by $A\in M$, i.e.\ there is $p\in\PP$ such that for every $q\leq_\PP p$ and every $\QQ\setel\PP$ in $M$, there is $a\in A$ so that \[\bar D_{q,a}=\{r\in\QQ^{\parallel q}\mid \forall s\leq_\PP q,r\,(s\leq_\PP^* a)\text{ or }\forall s\leq_\PP q,r\,(s\perp_\PP a)\}\] is not dense in $\QQ^{\parallel q}$. We want to use this assumption to find a $\PP$-generic filter $G$ over $\MM$ such that $\sigma^G$ does not lie in the induced $\QQ$-generic extension for any $\QQ\setel \PP$, i.e.\ not every new set is added by a set-sized complete subforcing.

We enumerate all dense subclasses of $\PP$ which are in $\C$ (from the outside) by $\langle D_n\mid n\in\omega\rangle$,
all $\QQ\setel\PP$ by $\langle\QQ_n\mid n\in\omega\rangle$ so that every $\QQ\setel\PP$ is enumerated unboundedly often, and we let $\langle \rho_n\mid n\in\omega\rangle$ be so that each $\rho_n$ is a $\QQ_n$-name for a subset of $A$, and so that for every $i\in\omega$, every $\QQ_i$-name $\rho$ 
%for a subset of $A$ 
is enumerated as some $\rho_n$.

Now we define a decreasing sequence of conditions $\langle q_n\mid n\in\omega\rangle$ below $p$ and a sequence
$\langle a_n\mid n\in\omega\rangle$ of conditions in $A$. Let $q_0=p$. Given $q_n\leq_\PP p$, we use our assumption to pick $a_n\in A$ such that $\bar D_{q_n,a_n}$ is not dense in $\QQ_n^{\parallel q_n}$. We may thus pick $r_0\in\QQ_n^{\parallel q_n}$ such that no $r_1\leq_{\QQ_n} r_0$ lies in this set. Pick $r_1\leq_{\QQ_n} r_0$ in $\QQ_n^{\parallel q_n}$ which decides whether or not $\check a_n\in\rho_n$. This can be done because if $B$ is a maximal antichain, of conditions below $r_0$ in $\QQ_n$ which decide whether or not $\check a_n\in\rho_n$, then $B$ is also maximal below $r_0$ in $\PP$, since $\QQ_n$ is a complete subforcing of $\PP$. In particular there must be $r_1\in B$ which is compatible with $q_n$ in $\PP$. 

Since $r_1\not\in\bar D_{q_n,a_n}$, we may now pick $\tilde q_n\leq_\PP q_n,r_1$ such that $\tilde q_n\perp_\PP a_n$ in case $r_1\Vdash_{\QQ_n}\check a_n\in\rho_n$, and such that $\tilde q_n\leq_\PP^* a_n$ in case $r_1\Vdash_{\QQ_n} \check a_n\not\in\rho_n$. Now take $q_{n+1}\leq_\PP\tilde q_n$ such that $q_{n+1}\in D_n$. 
In the end, this constructions yields a $\PP$-generic filter $G=\{q\in\PP\mid\exists n\in\omega\,(q_n\leq_\PP q)\}$. 
But since every new set added by $\PP$ is added by a set-sized complete subforcing by assumption, and since $\one_\PP\Vdash_\PP\sigma\subseteq\check{A}$, there must be some $n\in\omega$ such that $M[G]\models\sigma^G=\rho_n^{\bar G_n}$, where $\bar G_n=G\cap\QQ_n$. 
But either $q_{n+1}\Vdash_\PP\check a_n\in\rho_n$ and $a_n\bot_\PP q_{n+1}$, thus $a_n\notin\sigma^G$, or $q_{n+1}\Vdash_\PP\check a_n\notin\rho_n$ and $q_{n+1}\leq_\PP^* a_n$, implying that $a_n\in\sigma^G$. 
Thus $\sigma^G\ne\rho_n^{\bar G_n}$, and we have reached a contradiction. 
\end{proof}

Putting together Lemma \ref{srp1} and \ref{srp2} we obtain

\begin{corollary}%removed label
  If $\PP$ is a notion of class forcing such that every new set added by $\PP$ is already added by a set-sized complete subforcing of $\PP$, then $\PP$ satisfies the Forcing Theorem.
\end{corollary}

\section{Approachability by projections}\label{approachability}

In this short section, we want to generalize the property of the same name that was introduced in \cite[Section 6]{ClassForcing}. We want to use the very same name for this generalized property, as we think that this new property is what approachability by projections should have been defined as in the first place, while the property from \cite[Section 6]{ClassForcing} should perhaps be renamed as \emph{ordinal approachability by projections} (see our below remarks). We start by isolating a strong projection property (this very same property was already used in \cite[Section 6]{ClassForcing}). Note that in the below, (1)--(3) are the usual defining properties of a projection.

\begin{definition} 
Suppose that $\QQ$ is a subforcing of $\PP$ containing $X$ as a subset of its domain. A \emph{projection $\pi\colon \PP\rightarrow \QQ$ respecting $X$} is a function satisfying the following properties. 
\begin{enumerate}
  \item $\pi(\one_\PP)=\one_\PP$,
  \item $\forall p,q\in P\ (p\leq_\PP q\to\pi(p)\leq_{\PP}\pi(q))$,
  \item $\forall p\in P\,\forall q\leq_{\QQ}\pi(p)\,\exists r\leq_\PP p\ ( \pi(r)\leq_{\PP} q)$,
  \item $\forall p\in X\,\forall q\in P\ (\pi(q)\leq_{\PP} p\to q\leq_\PP p)$ and
  \item $\pi$ is the identity on $X$.
\end{enumerate}  
\end{definition} 

\begin{definition}\label{appproj}
Let $\MM\models\GB^-$ and let $\PP$ be a notion of class forcing for $\MM$. We say that $\PP=\langle P,\leq_\PP\rangle$ is \emph{approachable by projections} if there is a class $\Pi\in\C$ such that its sections $\pi_{X,y}=\{(u,v)\mid (X,y,u,v)\in\Pi\}$ have the following property. For all subsets $X$ of $\PP$, there is a set $y$ and a set-sized subforcing $\QQ$ of $\PP$ containing $X\cup\{\one_\PP\}$ as a subset such that $\pi_{X,y}\colon \PP\rightarrow \QQ$ is a projection respecting $X$.  
%Moreover, we require that the classes $\pi_{X,y}$ are coded by a single class, i.e. 
%$$\{\langle  X ,y, p, \pi_{X,y}(p)\rangle\mid X,y\in M,p\in P\}\in\C.$$ 
\end{definition}

\begin{lemma}\label{lemma:col appr}
  Let $\MM\models\GB^-$, let $\gamma\in\On^M$, let $Y$ be a proper class of $\MM$ and let $\PP=\Col(\gamma,Y)^M$. Then $\PP$ is approachable by projections.
\end{lemma}
\begin{proof}
  For a subset $X$ of $\PP$ in $M$ and $y\in M$, let $\range X:=\bigcup\{\range q\mid q\in X\}$ and let $\pi_{X,y}$ be trivial if $y\in\range X$ or if $y\not\in Y$, and otherwise let it map $p\in\PP$ to $\bar p\in\Col(\gamma,\range X\cup\{y\})$ by sending $p(i)$ to itself whenever $p(i)\in\range X\cup\{y\}$ and sending it to $y$ otherwise. Since $X$ is set-sized, there will be some $y\in Y$ such that $\pi_{X,y}$ is nontrivial. Verifying that these $\pi_{X,y}$ are projections respecting $X$ is now an easy exercise that we will leave to the reader.
\end{proof}

Approachability by projections in the sense of \cite[Section 6]{ClassForcing} is the special case when the set-size subforcings $\QQ$ of $\PP$ are always required to be of the form $\QQ_{\alpha+1}$ for an increasing sequence $\langle\QQ_\alpha\mid\alpha\in\On^M\rangle\in\C$ with union $\PP$. Our redefined property is strictly weaker than approachability by projections in the sense of \cite[Section 6]{ClassForcing}. For instance, the forcing notion $\Col(\omega,\mathcal P(\omega))$ is approachable by projections in the sense of this paper. It is not approachable by projections in the sense of \cite[Section 6]{ClassForcing} in a model of $\mathsf{ZFC}^-$ with the property that every set is countable and every set of reals has the property of Baire, since in such a model, there is no prewellordering of ${}^\omega2$ of length $\On$ whose equivalence classes are sets. This follows easily from the Kuratowski-Ulam theorem.  

%\todo[inline]{Add citation?}

\begin{lemma}\label{projectionpreservesforcing}
Suppose that $\pi\colon \PP\rightarrow \QQ$ is a projection respecting $\tc(\sigma\cup\tau)$ and $p\in \PP$. Then 
$$p\Vdash_\PP \sigma\subseteq \tau \Leftrightarrow \pi(p)\Vdash_\QQ \sigma\subseteq \tau.$$ 
\end{lemma} 
\begin{proof} 
  Like in the proof of \cite[Theorem 6.4]{ClassForcing}. We leave the checking of the details to the interested reader, as it boils down to a notational adaption of the original proof. Let us just say that essentially $P_\alpha$ needs to be replaced by $\tc(\sigma\cup\tau)$, $P_{\alpha+1}$ needs to be replaced by $\QQ$ and $\pi_{\alpha+1}$ by $\pi$ when adapting the proof.
\end{proof} 

\begin{corollary}\label{approachableft}
  If $\MM\models\GB^-$ and $\PP$ is a notion of class forcing for $\MM$ that is approachable by projections, then $\PP$ satisfies the forcing theorem over $\MM$.
\end{corollary}
\begin{proof}
  By Lemma \ref{projectionpreservesforcing}, we can define the forcing relation for equality, namely $p\Vdash_\PP\sigma=\tau$ iff there is a set $y$ and a set-sized subforcing $\QQ$ of $\PP$ such that $\pi=\pi_{\tc(\sigma\cup\tau),y}$ is a projection from $\PP$ to $\QQ$ that respects $\tc(\sigma\cup\tau)$ and such that $\pi(p)\Vdash_\QQ\sigma=\tau$. Note that the latter is definable for $\QQ$ is a set-sized notion of forcing. This suffices by \cite[Theorem 4.3]{ClassForcing}.
\end{proof}

\section{How not to turn proper classes into sets}\label{not turning classes into sets}

In this section, we will provide a collection of sufficient conditions ensuring that no proper class of the ground model turns into a set in a generic class forcing extension. This will be contrasted in the next section, where we provide a notion of class forcing that actually does turn a proper class into a set. A central notion in this context will be that of bounded and unbounded names.

\begin{definition}
  If $\MM\models\GB^-$ and $\PP$ is a notion of class forcing for $\MM$, then we call a $\PP$-name $\sigma$ \emph{bounded} if there is $A\in M$ such that $$\{\sigma^G\mid G\textrm{ is }\PP\textrm{-generic over }\MM\}\cap M\subseteq A.$$
We say that $\sigma$ is an \emph{unbounded} name otherwise.
We say that $\PP$ \emph{has bounded names} (over $\MM$) if there is no unbounded $\PP$-name $\sigma\in M$.
\end{definition}

This property is self-strengthening in the following sense.

\begin{lemma}
  Assume that $\MM\models\GB^-$ and $\PP$ is a notion of class forcing for $\MM$ that has bounded names. Assume further that $\langle\sigma_i\mid i\in I\rangle\in M$ is a sequence of $\PP$-names. Then there is a sequence $\langle A_i\mid i\in I\rangle$ such that $$\{\sigma_i^G\mid G\textrm{ is }\PP\textrm{-generic over }\MM\}\cap M\subseteq A_i$$ for every $i\in I$.
\end{lemma}
\begin{proof}
  Let $\langle\sigma_i\mid i\in I\rangle\in M$ be a sequence of $\PP$-names, such that $\sigma_i=\{\langle\tau_i^j,p_i^j\rangle\mid j\in J_i\}$ for every $i\in I$. Let $\sigma=\{\langle\op(\tau_i^j,\check i),\one_\PP\rangle\mid i\in I,j\in J_i\}$. Since $\sigma$ is a bounded name by assumption, we may find $A\in M$ such that $\{\sigma^G\mid G\textrm{ is }\PP\textrm{-generic over }\MM\}\cap M\subseteq A$. For every $i\in I$, let $A_i=\{a\mid\langle a,i\rangle\in A\}$. Then $\langle A_i\mid i\in I\rangle\in M$ is easily seen to be as desired.
\end{proof}

\begin{lemma}\label{boundednamescs}
  If $\MM\models\GB^-$ and $\PP$ is a notion of class forcing for $\MM$ that has bounded names, then no proper class $X$ of $\MM$ is turned into a set by forcing with $\PP$.
\end{lemma}
\begin{proof}
  Assume for a contradiction that $\sigma\in M$ is such that $\sigma^G=X$ for some proper class $X$ of $\MM$ and some $\PP$-generic filter $G$ over $\MM$. Write $\sigma$ as $\sigma=\{\langle\tau_i,p_i\rangle\mid i\in I\}$. By the bounded names property, we find $\langle A_i\mid i\in I\rangle$ such that for every $i\in I$ and every $\PP$-generic filter $H$ over $\MM$, $\tau_i^H\in A_i$. But this implies that $X=\sigma^G\subseteq\bigcup_{i\in I}A_i\in M$ for every $\PP$-generic filter $G$ over $\MM$. Using separation for the predicate $X$ in $\MM$, we obtain that $X\in M$, contradicting that $X$ is a proper class of $\MM$.
\end{proof}

\begin{lemma}
  If $\MM\models\GB$ and $\PP$ is a notion of class forcing for $\MM$, then $\PP$ has bounded names.
\end{lemma}
\begin{proof}
  Let $\sigma$ be a $\PP$-name in $M$, of rank less than $\alpha$. By a standard argument (in $\V$), we know that the rank of $\sigma^G$ is less than $\alpha$ for any $\PP$-generic filter $G$ over $\MM$. Hence if $\sigma^G\in M$, then $\sigma^G\in M_\alpha$, and $M_\alpha\in M$ by the power set axiom in $\MM$, hence $\sigma$ is a bounded name.
\end{proof}

\begin{lemma}
  If $\MM\models\GB^-$ and $\PP\in\MM$ is a notion of set forcing, then $\PP$ has bounded names.
\end{lemma}
\begin{proof}
  Let $\sigma$ be a $\PP$-name in $M$. Whenever $\sigma^G=z\in M$ for some $\PP$-generic filter $G$ over $\MM$, then there is $p\in\PP$ forcing that $\sigma=\check z$. Hence $\sigma$ is a bounded name by replacement in $\MM$.
\end{proof}

Next we consider approachability by projections.

\begin{lemma}\label{approachablebn}
  If $\MM\models\GB^-$ and $\PP$ is a notion of class forcing for $\MM$ that is approachable by projections, then $\PP$ has bounded names.
\end{lemma}
\begin{proof}
  Let $\sigma$ be a $\PP$-name in $M$. Using approachability by projections, let $\QQ\supseteq\tc(\sigma)$ be set-sized and let $\pi\colon\PP\to\QQ$ be a projection respecting $\tc(\sigma)$. Since $\PP$ satisfies the forcing theorem by [citation], for any possible value $z$ of $\sigma^G$ for some $\PP$-generic filter $G$ over $\MM$, there is $p\in\PP$ forcing that $\sigma=\check z$. By Lemma \ref{projectionpreservesforcing}, $\pi(p)\Vdash\sigma=\check z$. But $\pi(p)\in\QQ$, i.e.\ any possible value of $\sigma$ is decided by a condition in the set-sized forcing notion $\QQ$, so $\sigma$ is a bounded name by replacement in $\MM$.
\end{proof}

In Section \ref{srp}, we introduced the \emph{set reduction property} and showed that for a notion of class forcing $\PP$, this property is equivalent to the property that every new set added by $\PP$ is in fact added by a set-sized complete subforcing of $\PP$. We now show that this property ensures that no proper class is turned into a set.

\begin{lemma}
  If $\MM\models\GB^-$ and $\PP$ is a notion of class forcing for $\MM$ such that every new set is added by a set-sized complete subforcing of $\PP$, then no proper class $X$ of $\MM$ is turned into a set by forcing with $\PP$. 
\end{lemma}
\begin{proof}
  If $X$ were a new set in some $\PP$-generic extension of $\MM$, then it would have a $\QQ$-name $\sigma\in M$ for some set-sized complete subforcing $\QQ$ of $\PP$. Assume $\langle\tau,p\rangle\in\sigma$. Using the forcing theorem, which holds by [citation], any possible value of $\tau$ in a $\QQ$-generic extension is forced by some condition $q\in\QQ$. Using that $\QQ$ is set-sized, we can cover the possible values of $\tau$ by a set in $\MM$. But then we can cover $\sigma^G$ for any $\PP$-generic $G$ over $\MM$ by a single set in $\MM$. Pick $G$ such that $\sigma^G=X$. Using separation for the predicate $X$ implies that $X\in M$, contradicting that $X$ is a proper class of $\MM$.
\end{proof}

\begin{lemma}\label{lemma:pretame pres proper class}
  If $\MM\models\GB^-$ and $\PP$ is a pretame notion of class forcing for $\MM$, then no proper class $X$ of $\MM$ is turned into a set by forcing with $\PP$.
\end{lemma}
\begin{proof}
  Assume for a contradiction that $\sigma\in M$ is a $\PP$-name for some proper class $X$ of $\MM$. Note that $\PP$ satisfies the forcing theorem by \cite[Theorem 2.4]{pretameness}, so there is $p\in\PP$ such that $p\Vdash\sigma=\check X$. For $\langle\tau,r\rangle\in\sigma$, let $D_{\langle\tau,r\rangle}=\{d\in\PP\mid d\parallel\tau\}$. Applying pretameness of $\PP$, we may find $q\le p$ and $\langle d_{\langle\tau,r\rangle}\mid\langle\tau,r\rangle\in\sigma\rangle$ such that each $d_{\langle\tau,r\rangle}\subseteq D_{\langle\tau,r\rangle}$ is set-sized and predense below $q$. But this means that $q$ forces a set-sized ground model cover for $\sigma$, contradicting that $q\le p$ forces that $\sigma=\check X$.
\end{proof}

If $\MM\models\GB^-$ and thinks that each of its elements is countable, then there is a notion of class forcing for $\MM$ that is pretame and does not add any new sets, however has an unbounded name:

\begin{definition}\label{IntroP}
  Let $\PP$ denote the forcing whose conditions are (not necessarily finite) partial functions from $\omega$ to $2$, ordered by reverse inclusion -- equivalently, one may consider $\PP$ to be the full support iteration of length $\omega$ of the lottery sum of $\{0,1\}$, ordered naturally.
\end{definition}

\begin{lemma} \label{example for an unbounded name}
There is a $\PP$-name $\sigma$ such that for all $x\subseteq\omega$ in $M$, there is a $\PP$-generic filter $G$ over $M$ with $\sigma^G=x$. 
\end{lemma} 
\begin{proof} 
Let $\sigma=\{\langle\check{n},\{\langle n,1\rangle\}\rangle\mid n\in\omega\}$. Pick some $x\subseteq\omega$ in $M$. Since $x^*=\{\langle n,1\rangle\mid n\in x\}$ is an atom of $\PP$, $G=\{p\in\PP\mid p\subseteq x^*\}$ is a $\PP$-generic filter over $M$ that satisfies $\sigma^G=x$.
\end{proof} 

\begin{remark}
 If a notion of class forcing satisfies the forcing theorem, it does not necessarily have bounded names: A counterexample is provided by $\PP$ and $\MM$ above. $\PP$ satisfies the forcing theorem over $\MM$ by Lemma \ref{lem:no sets->sdp} and Lemma \ref{lem:ft}, since it does not add new sets. Lemma \ref{example for an unbounded name} shows that $\sigma$ is an unbounded $\PP$-name over $\MM$.
\end{remark}

\section{How to turn a proper class into a set}\label{turning classes into sets}

%$\mathsf{ZFC}^-$ denotes the theory $\mathsf{ZFC}$ without the power set axiom and with the collection scheme. 
%If $M$ is a model of $\mathsf{ZFC}^-$ with a hierarchy, $C$ is a proper class in $M$ \todo{do we need the forcing theorem?}and $G$ is class generic over $M$, then $C$ is not a set in $M[G]$. 

%We fix a well-ordering $\leq^*$ of $V_\omega$ with order type $\omega$ that is definable over $V_\omega$. Let $I_n$ denote the set of the first $n$ elements of $V_\omega$ with respect to $\leq^*$. 
%Suppose that $(x,m)$ and $(y,n)$ are pairs with $m,n\leq\omega$, $x$ is a subset of $I_m$ and $y$ is a subset of $I_n$. We say that $(x,m)$ is an \emph{initial segment} of $(y,n)$ if and only if $m\leq n$ and for all $z\in I_m$, $z\in x\Leftrightarrow z\in y$. 

We will show that over a model $M$ of $\ZF^-$ which thinks that all sets are countable, one can perform a fairly simple class forcing that turns the reals of $M$ (which are a proper class of $M$ by the proof of Cantor's diagonalization argument performed within $M$) into a set in its generic extensions.

 Let $\QQ$ be the finite support product of $\omega$-many copies of the notion of class forcing $\PP$ from Definition \ref{IntroP}.
We claim that forcing with $\QQ$ turns the reals of $M$ into a set in any of its generic extensions. Moreover, we will show that $\QQ$ does not satisfy the forcing theorem. %Using the ideas of this section one could provide a similar forcing notion and argument of proof over any model $M$ that fails to satisfy the power set axiom -- under our above assumptions on $M$, the power set of $\omega$ does not exist in $M$.
This provides easier alternative witnesses for \cite[Theorem 1.3]{ClassForcing}, that is notions of class forcing which fail to satisfy the forcing theorem, however only over certain models of $\ZF^-$. %, in the cases when the power set axiom fails in the ground model (\cite[Theorem 1.3]{ClassForcing} applies also when the ground model does satisfy the power set axiom). \todo[inline]{The above would need checking, maybe we should just omit claiming it.}%Let us first quickly observe that $\PP$, and therefore also $\QQ$, does not have bounded names.

\begin{lemma} \label{properties of the forcing Q}
  There is a $\QQ$-name $\tau$ such that for every $\QQ$-generic filter $G$ over $M$, \[\tau^G=\mathcal P(\omega)^M.\] 
\end{lemma} 
\begin{proof} 
For $i<\omega$, let $\tau_i=\{\langle\check{n},\{\langle\langle i,n\rangle,1\rangle\}\rangle\mid n\in\omega\}$, that is $\tau_i$ is the canonical $\QQ$-name for the real chosen in its $i^\textrm{th}$ iterand. Note that each $\tau_i$ is an unbounded name. Let $\tau=\{\langle\tau_i,1\rangle\mid i\in\omega\}$. For every $x\colon\omega\to 2$ in $M$, the set 
$$D_x=\{\langle p_0,\dots,p_k\rangle\in \QQ\mid k\in\omega\,\land\,\exists i\leq k\ p_i=x\}$$
is dense in $\QQ$. Hence $\tau^G=\mathcal P(\omega)^M$ for every $\QQ$-generic filter $G$ over $M$. 
\end{proof} 

$\Col(\omega,\mathcal P(\omega))^M$ is clearly isomorphic to a dense subforcing of $\QQ$. However we will close this section by showing that (unlike $\Col(\omega,\mathcal P(\omega))^M$, by Corollary \ref{approachableft}), $\QQ$ does not satisfy the forcing theorem (of course, $\Col(\omega,\mathcal P(\omega))^M$, being approachable by projections, also fails to turn a proper class of $M$ into a set, by Lemma \ref{approachablebn} and Lemma \ref{boundednamescs}).

\begin{theorem}\label{thm:failure ft}
  $\QQ$ does not satisfy the forcing theorem over $\MM$.
\end{theorem}
\begin{proof}
  Assume for a contradiction that $\QQ$ does satisfy the forcing theorem over $\MM$. We show that we can use this assumption to define a first order truth predicate over $\MM$ which will clearly be a contradiction. We will start by using this assumption to define a truth predicate for the two-sorted structure $\mathcal S=\langle\mathcal P(\omega)^M,\omega,\in,=,<,+,\cdot\rangle$ of second order arithmetic.

We will translate first order formulas over $\mathcal S$ into infinitary quantifier-free formulae in the forcing language of $\QQ$ so that truth over $\mathcal S$ of instances of the former corresponds to forcing respective instances of the latter. The infinitary language $\L_{\On,0}^{\Vdash}(\QQ,M)$ is built up from the atomic formulae $\check q\in\dot G$, $\sigma\in\tau$ and $\sigma=\tau$ for $q\in\QQ$ and $\sigma,\tau\in M^\PP$, the negation operator and set-sized conjunctions and disjunctions. This language originates from \cite[Section 5]{ClassForcing}, where also a more detailed description of this language may be found.

Now inductively, we assign to every first order formula $\varphi$ over $\mathcal S$ with free variables for natural numbers in $\{u_0,\dots,u_{k-1}\}$, free variables for reals in $\{v_0,\dots,v_{l-1}\}$ and all $\vec n=n_0,\dots,n_{k-1}\in\omega^k$ and $\vec\alpha=\alpha_0,\dots,\alpha_{l-1}\in\omega^l$ an $\L_{\On,0}^{\Vdash}(\QQ,M)$-formula in the following way. If $t$ is any $\mathcal S$-term, let $t(u_0,\dots,u_{k-1})_{\vec\alpha,\vec n}^*=t(\check{n_0},\dots,\check{n_{k-1}})$. For the sake of simplicity, from now on we only consider formulas that only involve trivial terms (and nontrivial terms would need to be handled as above).
\begin{align*}
 (u_i<u_j)_{\vec\alpha,\vec n}^*&=(\check{n_i}<\check{n_j})\\
 (u_i\in v_j)_{\vec\alpha,\vec n}^*&=(\check{n_i}\in\sigma_{\alpha_j})\\
 (\neg\varphi)_{\vec\alpha,\vec n}^*&=(\neg\varphi_{\vec\alpha,\vec n}^*)\\
 (\varphi\vee\psi)_{\vec\alpha,\vec n}^*&=(\varphi_{\vec\alpha,\vec n}^*\vee\psi_{\vec\alpha,\vec n}^*)\\
 (\exists v_k\varphi)_{\vec\alpha,\vec n}^*&=(\bigvee_{i<\omega}\varphi_{\vec\alpha^\frown i,\vec n}^*).
\end{align*}

Note that by \cite[Lemma 5.2]{ClassForcing},
if $\QQ$ satisfies the definability lemma for $\anf{v_0\in v_1}$ or $\anf{v_0=v_1}$, then it satisfies 
the uniform forcing theorem for all infinitary formulae in the forcing language of $\QQ$. %Therefore, it suffices to show that  the forcing theorem fails for some infinitary formula of the form $\varphi_{\vec\alpha}^*$. Suppose the contrary.
The following claim will thus allow us to define a truth predicate for first-order formulas over $\mathcal S$. %, contradicting Tarski's theorem on the 
%undefinability of truth

\begin{claim}\label{claim:truth}
 For every first-order formula $\varphi$ over $\mathcal S$ with free variables for natural numbers among $\{u_0,\dots,u_{k-1}\}$ and free variables for reals among $\{v_0,\dots,v_{l-1}\}$ and for all $\vec n=n_0,\dots,n_{k-1}\in\omega^k$ and all sequences of reals $\vec r=r_0,\dots,r_{l-1}$ in $M$,
 the following statements are equivalent:
 \begin{enumerate}
  \item $\mathcal S\models\varphi(\vec r,\vec n)$.
  \item $\forall\vec\alpha\in\omega^l\,\forall q\in\QQ\ q\Vdash_\QQ\anf{\forall i<l\ \sigma_{\alpha_i}=\check{r_i}}\ra q\Vdash_\QQ\varphi_{\vec\alpha,\vec n}^*$.
  \item $\exists\vec\alpha\in\omega^l\,\exists q\in\QQ\ q\Vdash_\QQ\anf{\forall i<l\ \sigma_{\alpha_i}=\check{r_i}}\wedge
 q\Vdash_\QQ\varphi_{\vec\alpha,\vec n}^*$.
 \end{enumerate}
\end{claim}

\begin{proof}
 Observe that since for any $\QQ$-generic filter $G$ over $\MM$, $\{\sigma_i^G\mid i<\omega\}=\mathcal P(\omega)^M$, (2) always implies (3). We will show the equivalence of (1), (2) and (3) by induction on formula complexity. For formulas of the form $\anf{u_i<u_j}$ this is obvious. Consider formulas of the form $\anf{u_i\in v_j}$. Suppose first that $n_i\in r_j$. Let $\alpha<\omega$ and $q\in\QQ$ with $q\Vdash_\PP\sigma_\alpha=\check{r_j}$.
Take a $\QQ$-generic filter with $q\in G$. But then obviously $q\Vdash_\PP\check{n_i}\in\sigma_\alpha$, i.e.\ (2) holds.

Assume now that (3) holds, i.e. there is $\alpha<\omega$ and $q\in\QQ$ such that $q\Vdash_\QQ\sigma_\alpha=\check{r_j}$ and $q\Vdash_\QQ\check{n_i}\in\sigma_\alpha$. Then $q\Vdash_\QQ\check{n_i}\in\check{r_j}$, i.e.\ $n_i\in r_j$ in $M[G]$ whenever $G$ is $\QQ$-generic over $\MM$, and by absoluteness, this statement holds true in $\mathcal S$, i.e.\ (1) holds.

The cases of negations and disjunctions are treated in a fairly standard way, exactly as in the proof of \cite[Theorem 2.6, Claim 2]{pretameness}.

%Next we turn to negations. Suppose first that $M\models\neg\varphi(\vec x)$ and let $\vec\alpha\in\kappa^k$ and $q\leq_\PP p$ with $q\Vdash_\PP\forall i<k\,(\dot F(\check\alpha_i)=\check x_i)$.
%Assume, towards a contradiction, that $q\nVdash_\QQ\neg\varphi_{\vec\alpha}^*$. Then there is $r\leq_\QQ q$
%with $r\Vdash_\QQ\varphi_{\vec\alpha}^*$. By density, we may assume that $r\in\PP$.
%Then $r\leq_\PP p$ and so $\vec\alpha$ and $r$ witness (3) for $\varphi$. By our inductive hypothesis we obtain that $M\models\varphi(\vec x)$,
%a contradiction. The implication from (3) to (1) is similar.

%Suppose now that $M\models(\varphi\vee\psi)(\vec x)$. Without loss of generality, assume that $M\models\varphi(\vec x)$.
%Now if $\vec\alpha\in\kappa^k$ and $q\leq_\PP p$ with $q\Vdash_\PP\forall i<k\,(\dot F(\check\alpha_i)=\check x_i)$, by induction
%$q\Vdash_\QQ\varphi_{\vec\alpha}^*$. But then in particular $q\Vdash_\QQ(\varphi\vee\psi)_{\vec\alpha}^*$. 
%In order to see that (3) implies (1), suppose that $\vec\alpha\in\kappa^k$ and $q\leq_\PP p$ witness (3). Then there must be a strengthenig $r\in\QQ$ of $q$ which satisfies, without loss of generality, $r\Vdash_\QQ\varphi_{\vec\alpha}^*$.
%By density of $\PP$ in $\QQ$, we can assume that $r\in\PP$. This means that $\vec\alpha$ and $r$ witness that (3)
%holds for $\varphi$, so $M\models\varphi(\vec x)$. 

We are thus left with the case of existential quantification. Assume first that $\mathcal S\models\exists v_l\varphi(\vec r^\frown v_l,\vec n)$. Pick $y\in\mathcal P(\omega)^M$ such that $\mathcal S\models\varphi(\vec x^\frown y,\vec n)$ and let $\vec\alpha\in\omega^l$ and $q\in\QQ$ with $q\Vdash_\PP\forall i<l\ \sigma_{\alpha_i}=\check{r_i}$. Let $G$ be $\QQ$-generic with $q\in G$. Then there is $t\leq_\QQ q$ and $\beta<\omega$ with $t\in G$ and $t\Vdash_\PP\sigma_\beta=\check y$. By induction, $r\Vdash_\QQ\varphi_{\vec\alpha^\frown\beta,\vec n}^*$.
In particular, $M[G]\models(\exists v_k\varphi)_{\vec\alpha,\vec n}^*$. The converse follows in a similar way. 
\end{proof}

%Let $\Fm_1$ denote the set of all G\"odel codes of $\L_\in$-formulae whose only free variable is $v_0$. 
%As a consequence of Claim \ref{claim:truth}, the class
%$$T=\{\langle\gbr{\varphi},x\rangle\mid\gbr\varphi\in\Fm_1\wedge x\in M\wedge\forall\alpha<\kappa\,\forall q\leq_\PP p\ q\Vdash_\PP\dot F(\check\alpha)=\check x\ra q\Vdash_\QQ\varphi_\alpha^*\}$$
%defines a first-order truth predicate for $M$, contradicting our assumptions on $\MM$. 
Using the above claim together with the assumption of the forcing theorem, we obtain a truth predicate for $\mathcal S$ definably over $M$. But then by the usual translation between $H_{\omega_1}$ and the reals, we obtain from this a truth predicate for $M$ that is definable over $M$, contradicting Tarski's undefinability of truth.
\end{proof}

\cite[Theorem 1.12]{pretameness} gives a list of many desirable properties of notions of class forcing which are -- under additional assumptions on the ground model -- equivalent to pretameness. For example, pretameness can be characterized in terms of the forcing theorem and the existence of a Boolean completion. In order to state these equivalences, we need the following.

\begin{notation}
	Let $\MM\models\GB^-$ and let $\Psi$ be some property of a notion of class forcing $\PP$ for $\MM=\langle M,\C\rangle$. We say that $\PP$ \emph{densely} satisfies $\Psi$ if every notion of class forcing $\QQ$ for $\MM$, for which there
	is a dense embedding in $\C$ from $\PP$ into $\QQ$, satisfies the property $\Psi$.
\end{notation}

\cite[Theorem 1.12]{pretameness} for example states that -- under certain conditions on the ground model -- a forcing notion is pretame iff it densely satisfies the forcing theorem. 
The following result yields yet another charachterization of pretameness of this kind.

\begin{lemma}\label{lemma:non-pretame class->set}
Suppose that $\MM\models\GBC^-$ such that $M$ contains a largest cardinal $\kappa$. If $\PP$ is a notion of class forcing for $\MM$ which is non-pretame but satisfies the forcing theorem, then there is a notion of class forcing $\QQ$ such that $\PP$ is dense in $\QQ$ and forcing with $\QQ$ turns a proper class into a set.
\end{lemma}

\begin{proof}
Suppose that $\PP$ is a non-pretame notion of class forcing for $\MM$. By \cite[Lemma 2.7]{pretameness}, the class of all $p\in\PP$ such that there is an ordinal $\alpha$ and a class name $\dot F$ with
$p\Vdash_\PP\anf{\dot F:\check{\alpha}\ra\On^M}\text{is surjective}"$ is dense. Using the existence of a global well-order of order-type $\On^M$, we may assume that $\one_\PP\Vdash_\PP\anf{\dot F:\check{\alpha}\ra \mathcal P(\kappa)^M}\text{is surjective}"$. Now we extend $\PP$ to a forcing notion $\QQ$ by formally adding 
the suprema of the classes
\begin{align*}
D_{\beta,\lambda}&=\{p\in\PP\mid p\Vdash_\PP\check\lambda\in\dot F(\check\beta)\}.
\end{align*}
for $\beta<\alpha$ and $x\in \mathcal P(\kappa)$. More precisely, let $\QQ=\PP\cup\{p_{\beta,\lambda}\mid\beta<\alpha,\lambda<\kappa\}$, where each $p_{\beta,\lambda}$ does not lie in $\PP$. We can then order $\QQ$ by
\begin{align*}
 p_{\beta,\lambda}\leq_{\QQ}p&\Llr\forall q\in D_{\beta,\lambda},( q\leq_\PP p),\\
 p\leq_{\QQ} p_{\beta,\lambda}&\Llr D_{\beta,\lambda}\text{ is predense below $p$ in }\PP,\\
 p_{\beta,\lambda}\leq_{\QQ} p_{\beta',\lambda'}&\Llr\forall q\in D_{\beta,\lambda}\,( q\leq_{\QQ} p_{\beta',\lambda'})
\end{align*}
for $p\in\PP$ and $\beta,\beta'<\alpha$ and $\lambda,\lambda'<\kappa$.
By construction, $\PP$ is a dense subforcing of $\QQ$. For $\beta<\alpha$ we define
\begin{align*}
 \sigma_\beta&=\{\langle\check\lambda,p_{\beta,\lambda}\rangle\mid\lambda<\kappa\}\\
 \sigma&=\{\langle\sigma_\beta,\one_\PP\rangle\mid\beta<\alpha\}.
\end{align*}
Now let $G$ be $\QQ$-generic over $\MM$. 

\begin{claim*}
 $\sigma^G=\mathcal P(\kappa)^M$.
\end{claim*}
\begin{proof}
Let $\beta<\alpha$. Then $\lambda\in \sigma_\beta^G$ iff $p_{\beta,\lambda}\in G$ iff $\lambda\in\dot F^G(\beta)$. Hence $\sigma_\beta^G=\dot F^G(\beta)\in\mathcal P(\kappa)^M$. This proves the claim, since $\dot F^G$ is surjective.
\end{proof}
Since $\mathcal P(\kappa)^M$ is a proper class in $\MM$, it follows from the claim above that $\QQ$ turns a proper class into a set.
\end{proof}

Note that a similar argument as the one given in the proof of Theorem \ref{thm:failure ft} shows that the forcing notion $\QQ$ in the proof of Lemma \ref{lemma:non-pretame class->set} does not satisfy the forcing theorem.
Using Lemmata \ref{lemma:pretame pres proper class} and \ref{lemma:non-pretame class->set} we obtain the following characterization of pretameness:

\begin{theorem}\label{thm:char pretame}
Suppose that $\MM\models\GBC^-$ such that $M$ contains a largest cardinal $\kappa$. Then a notion of forcing for $\MM$ which satisfies the forcing theorem is pretame for $\MM$ if and only if it densely does not turn proper classes into sets. \hfill\qedsymbol
\end{theorem}

\section{Open Questions}\label{questions}

In Section \ref{not turning classes into sets}, we show that all known properties of forcing notions which imply the forcing theorem (except, of course, the forcing theorem itself) prevent proper classes from being turned into sets in generic extensions. A natural question is therefore the following:

\begin{question}
 Does the forcing theorem imply that no proper class in the ground model is turned into a set in the generic extension?
\end{question}

By now, we know a range of combinatorial properties that imply the forcing theorem to hold, however we do not know a combinatorial characterization of the forcing theorem itself.

\begin{question}
  Is there a combinatorial property that holds for a notion of class forcing $\PP$ exactly if $\PP$ satisfies the forcing theorem?
\end{question}

%In the proof of Theorem \ref{thm:char pretame} we require the existence of a set-like well-order of the ground model. It would therefore be interesting to know whether this assumption is necessary.

%\begin{question}
%Can pretameness be characterized in terms of the property of not turning proper classes into sets in $\GB^-$?
%\end{question}

\bibliographystyle{alpha}
\bibliography{class-forcing}

\newcommand{\etalchar}[1]{$^{#1}$}
\begin{thebibliography}{HKL{\etalchar{+}}16}

\bibitem[Fri00]{MR1780138}
Sy~D. Friedman.
\newblock {\em {Fine structure and class forcing}}, volume~3 of {\em {de
  Gruyter Series in Logic and its Applications}}.
\newblock Walter de Gruyter \& Co., Berlin, 2000.

\bibitem[Git13]{gitman}
Victoria Gitman.
\newblock {Forcing to add proper classes to a model of GBC}.
\newblock {\em Blog Post},
  http://boolesrings.org/victoriagitman/2013/10/09/forcing-to-add-proper-classes-to-a-model-of-rm-gbc-an-introduction/,
  2013.

\bibitem[HKL{\etalchar{+}}16]{ClassForcing}
Peter Holy, Regula Krapf, Philipp L\"ucke, Ana Njegomir, and Philipp Schlicht.
\newblock Class forcing, the forcing theorem and {B}oolean completions.
\newblock {\em J. Symb. Log.}, 81(4):1500--1530, 2016.

\bibitem[HKS17]{pretameness}
Peter Holy, Regula Krapf, and Philipp Schlicht.
\newblock {Characterizations of Pretameness and the Ord-cc}.
\newblock Accepted for the Annals of Pure and Applied Logic, 2017.

\bibitem[Sta]{stanley_classforcing}
Maurice~C. Stanley.
\newblock Class forcing.
\newblock Incomplete draft, 1996, unpublished.

\bibitem[Sta84]{stanley_thesis}
Maurice~C. Stanley.
\newblock {\em A unique generic real}.
\newblock PhD thesis, U.\ C.\ Berkeley, 1984.

\bibitem[Zar73]{MR0345819}
Andrzej Zarach.
\newblock {Forcing with proper classes}.
\newblock {\em Fund. Math.}, 81(1):1--27, 1973.
\newblock Collection of articles dedicated to Andrzej Mostowski on the occasion
  of his sixtieth birthday, I.

\end{thebibliography}
  
\end{document}